\documentclass{amsart}
\usepackage{amsthm}
\usepackage{amsmath}
\usepackage{amssymb}
\usepackage{euscript}
\usepackage{graphicx}
\usepackage{pgf,tikz}
\usetikzlibrary{arrows}

\newtheorem{theorem}{Theorem}

\newtheorem{corollary}[theorem]{Corollary}
\newtheorem{proposition}[theorem]{Proposition}
\newtheorem{lemma}[theorem]{Lemma}
\newtheorem{remark}[theorem]{Remark}
\newtheorem{question}[theorem]{Question}
\newtheorem{definition}[theorem]{Definition}

\newcommand{\seq}[1]{\langle #1\rangle}

\makeatletter
\@namedef{subjclassname@2020}{%
	\textup{2020} Mathematics Subject Classification}
\makeatother

\begin{document}

\title[Shadowing as a Structural Property]{Shadowing as a Structural Property of the Space of Dynamical Systems}

\author[J. Meddaugh]{Jonathan Meddaugh}
\address[J. Meddaugh]{Baylor University, Waco TX, 76798}
\email[J. Meddaugh]{Jonathan\_Meddaugh@baylor.edu}


\begin{abstract} We demonstrate that there is a large class of compact metric spaces for which the shadowing property can be characterized as a structural property of the space of dynamical systems. We also demonstrate that, for this class of spaces, in order to determine whether a system has shadowing, it is sufficient to check that \emph{continuously generated} pseudo-orbits can be shadowed.
\end{abstract}

\maketitle

\section{Introduction}

For a function $f:X\to X$ on a metric space $X$, a $\delta$-pseudo-orbit is a sequence $\seq{x_i}$ in $X$ which satisfies $d(f(x_i),x_{i+1})<\delta$. These approximations of orbits arise naturally when orbits are computed numerically and can diverge significantly from the orbits that they are approximating. However, there is a large class of functions for which sufficiently precise pseudo-orbits can be tracked within a given level of tolerance. Such systems are said to have the \emph{shadowing property} (or \emph{pseudo-orbit tracing property}). It is immediately clear that this property is important in numerical modeling of dynamical systems, especially for those in which errors are likely to grow (see \cite{Corless, Corless2, palmer, Pearson}).

While it has natural applications in the setting of numerical modeling, shadowing has significant theoretical importance. Indeed, the origin of the shadowing property begins with analysis of of Anosov and Axiom A diffeomorphisms--Bowen implicitly used this property in his proof that the nonwandering set of an Axiom A diffeomorphism is a factor of a subshift of finite type \cite{Bowen}. Later, Sinai and Bowen demonstrated explicitly that the shadowing property holds for large subsystems of these diffeomorphisms \cite{bowen-markov-partitions, sinai}. Since then, it has been observed that shadowing plays an important role in stability theory \cite{Pil, robinson-stability, walters} and in characterizing $\omega$-limit sets \cite{BGOR-DCDS, Bowen,  MR}. Shadowing has also been shown to be a relatively common property in the space of dynamical systems on certain classes of spaces \cite{BMR-Dendrites, Mazur-Oprocha, Meddaugh-Genericity,  Mizera, Odani, Pilyugin-Plam,  Yano}.

In this paper, we demonstrate that, for a large class of spaces, the shadowing property is a \emph{structural property} of the space of dynamical systems ($X^X$ or $\mathcal C(X)$), in that whether a system has shadowing can be observed only by examining its neighborhood basis in that space. In particular, the main results of this paper are the following theorems.

{
	\renewcommand*{\thetheorem}{\ref{shadowing in X^X}}
	\addtocounter{theorem}{-1}
	\begin{theorem}
		Let $X$ be a compact metric space and $f\in \mathcal C(X)$. The dynamical system $(X,f)$ has shadowing if and only if for every $\epsilon>0$ there exists $\delta>0$ such that if  $g\in X^X$ with $\rho(g,f)<\delta$, then each $g$-orbit  is $\epsilon$-shadowed by an $f$-orbit. 
	\end{theorem}
}

{
	\renewcommand*{\thetheorem}{\ref{perturbable}}
	\addtocounter{theorem}{-1}
	\begin{theorem}
	Let $X$ be a compact metric space which is perturbable and let $f\in C(X)$. The dynamical system $(X,f)$ has shadowing if and only if for every $\epsilon>0$ there exists $\delta>0$ such that if $g\in \mathcal C(X)$ with $\rho(g,f)<\delta$, then each $g$-orbit  is $\epsilon$-shadowed by an $f$-orbit.
	\end{theorem}
}

The organization of this paper is as follows. In Section \ref{Prelim} we formalize the definitions and notation that we will use for the remainder of the paper. In Section \ref{X^X}, we focus on the space $X^X$ of not necessarily continuous dynamical systems and develop the theory leading up to Theorem \ref{shadowing in X^X}. We also prove results for nonautonomous dynamical systems. Section \ref{C(X)} discusses analogous results for the space $\mathcal C(X)$ and culminates in the proof of Theorem \ref{perturbable}. Finally, in Section \ref{other} we discuss the results of Sections \ref{X^X} and \ref{C(X)} from different angles. We develop variations of the shadowing property that encapsulate these results and also demonstrate that shadowing can be viewed as a weak form of continuity for the function which maps a dynamical system to its set of orbits.

\section{Preliminaries} \label{Prelim}

A \emph{dynamical system} is a pair $(X,f)$ consisting of a compact metric space $X$ and a function $f:X\to X$. This function need not be continuous, and is not assumed to be unless otherwise stated. 

For a fixed space $X$, we can identify a dynamical system with the function alone, allowing us to consider the space of dynamical systems on $X$. In particular, let $X^X$ denote the space of functions from $X$ to $X$ with topology generated by the supremum metric
\[\rho(f,g)=\sup_{x\in X}d(f(x),g(x)).\]
Also, let $\mathcal C(X)$ denote the subspace of $X^X$ consisting of the continuous functions (maps) from $X$ to $X$.

For $x\in X$ and $f\in X^X$, the \emph{orbit of $x$ under $f$} (or, the $f$-orbit of $x$) is the sequence $\seq{f^i(x)}_{i\in\omega}$ where $f^0$ denotes the identity. For $\delta>0$, a \emph{$\delta$-pseudo-orbit for $f$} is a sequence $\seq{x_i}_{i\in\omega}$ such that for all $i\in\omega$, $d(f(x_i),x_{i+1})<\delta$. For $\epsilon>0$, we say that a sequence $\seq{x_i}_{i\in\omega}$ in $X$ \emph{$\epsilon$-shadows} a sequence $\seq{y_i}_{i\in\omega}$ in $X$ provided that $d(x_i,y_i)<\epsilon$ for all $i\in\omega$.

A dynamical system $(X,f)$ has \emph{shadowing} (or the \emph{pseudo-orbit tracing property}, sometimes denoted \emph{POTP}) provided that for all $\epsilon>0$, there exists $\delta>0$ such that if $\seq{x_i}$ is a $\delta$-pseudo-orbit for $f$, then there exists $z\in X$ such that for all $i$, $d(x_i,f^i(z))<\epsilon$, i.e. the orbit of $z$ \emph{$\epsilon$-shadows} the pseudo-orbit $\seq{ x_i}$. 

For the purposes of this paper, it is useful to consider dynamical systems consisting of \emph{sequences} of functions from $X$ to $X$. When a distinction is necessary, we will refer to dynamical systems consisting of a single function as \emph{autonomous} and those consisting of several functions as \emph{nonautonomous}. We will endow the space, $(X^X)^\omega$, of such systems with the uniform topology, with metric given by
\[\rho(\seq{f_i},\seq{g_i})=\sup_{i\in\omega}\rho(f_i,g_i).\]

Let $\seq{f_i}\in (X^X)^\omega$. For $x\in X$, the \emph{orbit of $x$ under $\seq{f_i}$} (or, the $\seq{f_i}$-orbit of $x$) is the sequence $\seq{f^i_0(x)}$ where $f^i_0$ denotes the composition of the first $i$-many functions in the sequence, i.e. $f^i_0=f_{i-1}\circ f_{i-2}\circ\cdots\circ f_1\circ f_0$ and $f^0_0$ denotes the identity. For $\delta>0$, a \emph{$\delta$-pseudo-orbit for $\seq{f_i}$} is a sequence $\seq{x_i}$ such that for all $i\in\omega$, $d(f_i(x_i),x_{i+1})<\delta$.

We say that a nonautonomous dynamical system $(X,\seq{f_i})$  has \emph{shadowing} provided that for all that for all $\epsilon>0$, there exists $\delta>0$ such that if $\seq{x_i}$ is a $\delta$-pseudo-orbit for $\seq{f_i}$, then there exists $z\in X$ such that for all $i$, $d(x_i,f^i_0(z))<\epsilon$, i.e. the orbit of $z$ \emph{$\epsilon$-shadows} the pseudo-orbit $\seq{ x_i}$.

It is worth noting that the autonomous dynamical system $(X,f)$ can be naturally identified with the dynamical system $(X,\seq{f})$ since it is clear that the dynamical properties of the systems $(X,f)$ and $(X,\seq{f})$ are identical.

\section{Shadowing as a structural property in $X^X$} \label{X^X}

In this section, we examine shadowing as a structural property of dynamical systems in the \emph{space} of (not necessarily continuous) dynamical systems. We begin by developing a characterization of the shadowing property in the space $(X^X)^\omega$ of nonautonomous dynamical systems on $X$. 

As a first step, we make the following observation:

\begin{remark} \label{nearby}
	If $\seq{f_i},\seq{g_i}\in (X^X)^\omega$ with $\rho(\seq{f_i},\seq{g_i})<\delta$, then every $\seq{g_i}$-orbit is a $\delta$-pseudo-orbit for $\seq{f_i}$.
\end{remark} 

\begin{proof}
	Fix $\seq{f_i},\seq{g_i}$ with $\rho(\seq{f_i},\seq{g_i})<\delta$ and let $x\in X$. Then $\seq{g^i_0(x)}$ satisfies
	\[d(f_i(g^i_0(x)),g^{i+1}_0(x))=d(f(g^i_0(x)),g_i(g^{i}_0(x)))\leq\rho(f_i,g_i)\leq\rho(\seq{f_i},\seq{g_i})<\delta\]
	and is, therefore, a $\delta$-pseudo-orbit for $\seq{f_i}$. \end{proof}

In other words, for a dynamical system, orbits of nearby systems are pseudo-orbits for the original system. As we see in the following lemma, if $\seq{f_i}\in (X^X)^\omega$, the converse is true as well.

\begin{lemma}\label{discontinuous}
	Let $\seq{f_i}\in (X^X)^\omega$, $\delta>0$, and $\seq{x_i}$ be a $\delta$-pseudo-orbit for $\seq{f_i}$. Then there exists $\seq{g_i}\in (X^X)^\omega$ such that $\rho(\seq{f_i},\seq{g_i})\leq\delta$ and $\seq{x_i}$ is an orbit for $\seq{g_i}$.
\end{lemma}

\begin{proof}
	For each $i\in\omega$, define $g_i$ as follows: $g_i(y)=f_i(y)$ for all $y\in X\setminus\{x_i\}$ and $g_i(x_i)=x_{i+1}$.	By construction, the sequence $\seq{x_i}$ is an orbit of $\seq{g_i}$.
	
	Finally, it is easily seen that for each $i\in\omega$, $\rho(f_i,g_i)=d(x_{i+1},f_i(x_i))<
	\delta$, and thus $\rho(\seq{f_i},\seq{g_i})\leq\delta$. 
\end{proof}

Since pseudo-orbits are identifiable with orbits for nearby systems in this setting, we can easily verify that the shadowing property can be characterized as follows.

\begin{theorem} \label{shadowing in sequences}
	Let $X$ be a metric space and $\seq{f_i}\in (X^X)^\omega$. The dynamical system $(X,\seq{f_i})$ has shadowing if and only if for all $\epsilon>0$, there exists $\delta>0$ such that if $\seq{g_i}\in(X^X)^\omega$ with $\rho(\seq{g_i},\seq{f_i})<\delta$, then each $\seq{g_i}$-orbit  is $\epsilon$-shadowed by an $\seq{f_i}$-orbit.
\end{theorem}

\begin{proof}
	First, suppose that the dynamical system $(X,\seq{f_i})$ has shadowing. Fix $\epsilon>0$ and choose $\delta>0$ to witness shadowing with respect to $\epsilon$. If $\seq{g_i}$ is a sequence with $\rho(\seq{g_i},\seq{f_i})<\delta$, then every $\seq{g_i}$-orbit is a $\delta$-pseudo-orbit for $\seq{f_i}$, and is thus shadowed by some $\seq{f_i}$-orbit.
	
	Now, suppose that $(\seq{f_i},X)$ has the indicated property. Fix $\epsilon>0$ and choose $\delta_0>0$ to witness. Define $\delta=\delta_0/2$ and let $\seq{x_i}$ be a $\delta$-pseudo-orbit for $\seq{f_i}$. Then, by Lemma \ref{discontinuous}, there exists $\seq{g_i}$ with $\rho(\seq{f_i},\seq{g_i})\leq\delta<\delta_0$ such that $\seq{x_i}$ is a $\seq{g_i}$-orbit, and as such, is $\epsilon$-shadowed by some $\seq{f_i}$-orbit.
\end{proof}

Since autonomous systems can be viewed as members of $(X^X)^\omega$, as an immediate corollary,  we have the following characterization of shadowing for autonomous systems.

\begin{corollary} \label{shadowing in sequences corollary}
	Let $X$ be a metric space and $f\in X^X$. The dynamical system $(X,f)$ has shadowing if and only if for all $\epsilon>0$, there exists $\delta>0$ such that if $\seq{g_i}\in(X^X)^\omega$ with $\rho(\seq{g_i},\seq{f})<\delta$, then each $\seq{g_i}$-orbit  is $\epsilon$-shadowed by an $f$-orbit. 
\end{corollary}

Of course, this characterization requires knowledge of an entire neighborhood of $\seq{f}$ in $(X^X)^\omega$. For an autonomous dynamical system, a characterization requiring only information in $X^X$ would be far preferable. The primary obstacle to such a characterization is that a non-periodic pseudo-orbit which visits the same point twice cannot be realized as an orbit of an autonomous system---in order to guarantee shadowing, such pseudo-orbits need to be accounted for.

If we assume that our domain space contains no isolated points, we are able to bypass this obstacle for continuous dynamical systems by perturbing such pseudo-orbits. This will allow us to characterize shadowing as a structural property for autonomous dynamical systems on spaces with no isolated points.

\begin{lemma} \label{functionally generated}
	Let $f\in\mathcal C(X)$ and $\gamma>0$. Then if $\seq{x_i}$ a $\gamma$-pseudo-orbit for $f$ for which $x_i=x_j$ implies that $x_{i+1}=x_{j+1}$, then  there exists $g\in X^X$ with $\rho(g,f)\leq\gamma$ and $z\in X$  such that for all $i\leq N$, $g^i(z)=x_i$.
\end{lemma}

\begin{proof}
	Fix $\gamma>0$ and let $\seq{x_i}$ be a $\gamma$-pseudo-orbit for $f$ satisfying the hypotheses of the lemma. 
	
	Now, define $g:X\to X$ by
	\[g(x)=\begin{cases} 
		x_{i+1} & x=x_i \\
		f(x) & x\notin\{x_i: i\in\omega\}
	\end{cases}.\]
	Note that $g$ is well-defined since if $x_i=x_j$, then $g(x_i)=x_{i+1}=x_{j+1}=g(y_i)$. 
	
	By construction, $g(x_i)=x_{i+1}$ and for all $x\in X$, $d(f(x),g(x))\leq\sup\{d(f(x_i),x_{i+1}):i\in\omega\}\leq\gamma$, i.e. $\rho(f,g)\leq\gamma$.
\end{proof}

\begin{lemma} \label{periodic or not}
	Let $X$ be a compact metric space and let $f\in C(X)$ such that for every $\epsilon>0$ there exists $\delta>0$ such that if $g\in  X^X$ with $\rho(g,f)<\delta$, then each $g$-orbit  is $\epsilon$-shadowed by an $f$-orbit. If $p\in X$ is isolated, then for every $\delta'$ there exists $\gamma_p>0$ such that if $\seq{x_i}$ is a $\gamma_p$-pseudo-orbit for $f$, then either
	\begin{enumerate}
		\item $|\{i\in\omega:p=x_i\}|\leq 1$ or
		\item $p$ is periodic and there exists $M\in\omega$ with $x_M$ in the orbit of $p$ and for any such $M$, $\sup \{d({x_{M+i}},{f^i(x_M)})\}<\delta'$.
	\end{enumerate}
	
\end{lemma}

\begin{proof}
	Fix $\delta'>0$. Let $p\in X$ be isolated and suppose that for every $\gamma>0$ there exists a $\gamma$-pseudo-orbit for $f$ in which $p$ appears at least twice (else we can choose $\gamma_p>0$ such that condition (1) holds for every $\gamma_p$-pseudo-orbit).
	Since $p$ is isolated, we can choose $\epsilon>0$ with $B_\epsilon(p)=\{p\}$, and by supposition, there is $\delta>0$ such that if $\rho(g,f)<\delta$, then every $g$-orbit is $\epsilon$-shadowed by an $f$-orbit. Finally, choose $\delta>\gamma>0$.
	
	By supposition, we can find a $\gamma$-pseudo-orbit $\seq{x_i}$ for $f$ in which $p$ appears at least twice---let $M\in\omega$ and $N>0$ be minimally chosen so that $p=x_M=x_{N+M}$.
	
	Define, for $i<N$, $y_i=x_{M+i}$ and consider the finite sequence $y_0,\ldots, y_{N-1}$. Notice that $y_0=p$, $d(f(y_{N-1}),y_0)<\gamma$, and for all $i<N-1$, we have $d(f(y_i),y_{i+1})<\gamma$. Suppose that there is an element of $X$ that occurs more than once in this list. Choose $n<N-1$ and $t>0$ such that $y_n=y_{n+t}$ and define the finite sequence $y'_0,\ldots y'_{N'-1}$ as follows. Define $N'=N-t$ and 
	\[ y'_i=\begin{cases} 
		y_i & i\leq n \\
		y_{i+t} & i>n 
	\end{cases}
	\]
	and observe that $y'_0=p$, $d(f(y'_{N-1}),z_0)<\gamma$, and  for all $i<N'-1$, we have $d(f(y'_i),y'_{i+1})<\gamma$. By iterating this process as necessary, we arrive at a finite sequence $z_0,\ldots z_{P-1}$ with no repeated elements, such that $z_0=p$, $d(f(z_{P-1}), z_0)<\gamma$, and for $i<P-1$, $d(f(z_{i}), z_{i+1})<\gamma$. It is easy to see then that if we define $z_i=z_{i\mod P}$ for $i\geq P$,  $\seq{z_i}$ is a $\gamma$-pseudo-orbit with $z_{kP}=p$ for all $k\in\omega$ which satisfies the hypotheses of Lemma \ref{functionally generated}.
	
	As such, we can find a function $g\in X^X$ with $\rho(g,f)\leq\gamma<\delta$ such that $\seq{z_i}$ is a $g$-orbit, and therefore is $\epsilon$-shadowed by the $f$-orbit of some point. But that point and its image under $f^P$ must lie in $B_\epsilon(p)=\{p\}$, and therefore the shadowing point is $p$ itself and $p$ has period $P$.
	
	Now, by uniform continuity of $f, f^2, \ldots f^{P-1}$, choose $\gamma_p>0$ such that $\gamma_p<\min\{\epsilon,\delta'/2\}/P$ and if $d(a,b)<\gamma_p$, then $d(f^i(a),f^i(b))<\min\{\epsilon,\delta'/2\}/P$ for all $i<P$. Now, let $\seq{x_i}$ be a $\gamma_p$-pseudo-orbit for $f$ with $x_M$ in the orbit of $p$ with $k<P$ chosen such that $f^k(x_M)=p$. Then, for $i\leq P$, we have 
	\begin{align*}
		d(x_{M+i},f^i(x_M))&\leq \sum_{j=0}^{i-1} d\left(f^{(i-j)}(x_{M+j}),f^{(i-j-1)}(x_{M+j+1})\right)\\
		&=d(f(x_{M+i-1}),x_{M+i})+\sum_{j=0}^{i-2} d\left(f^{(i-j)}(f(x_{M+j})),f^{(i-j)}(x_{M+j+1})\right)\\
		&<\gamma_p+\sum_{j=0}^{i-2} \min\{\epsilon,\delta'/2\}/P<\min\{\epsilon,\delta'/2\}
	\end{align*}

	In particular, $x_{M+k}\in B_\epsilon(f^k(x_M))=B_\epsilon(p)=\{p\}$ and by induction (applying the above argument to $x_{M+k}=p$), we can establish that $d(x_{M+k+i},f^i(p))<\delta'/2$ for all $i\in\omega$, and thus $\sup\{d({x_{M+i}},f^i(p))\}\leq\delta'/2<\delta'$ as desired.

\end{proof}

Lemmas \ref{functionally generated} and \ref{periodic or not} allow us to prove the following characterization of shadowing.

\begin{theorem} \label{shadowing in X^X}
	Let $X$ be a compact metric space and $f\in \mathcal C(X)$. The dynamical system $(X,f)$ has shadowing if and only if for every $\epsilon>0$ there exists $\delta>0$ such that if  $g\in X^X$ with $\rho(g,f)<\delta$, then each $g$-orbit  is $\epsilon$-shadowed by an $f$-orbit. 
\end{theorem}

\begin{proof}
	As with Theorem \ref{shadowing in sequences}, if $(X,f)$ has shadowing, then the conclusion follows easily.
	
	So, let us now assume that $(X,f)$ has the indicated property and fix $\epsilon>0$. Choose $\epsilon/2>\delta>0$ to witness that if $g
	\in X^X$ with $\rho(f,g)<\delta$, then each $g$-orbit is $\epsilon/2$ shadowed by an $f$-orbit. Finally, choose $\delta>\gamma>0$.

	Now, by uniform continuity of $f$, fix $\gamma/3>\beta$ such that if $d(a,b)<\beta$, then $d(f(a),f(b))<\gamma/3$, and define
	\[S=\{x\in X:B_\beta(x) \textrm{ is infinite}\}.\]
	
	Observe that $X\setminus S$ is finite---otherwise, it would have a limit point $p$ and a point $q\in X\setminus S$ with $p\in B_\beta(q)$. But then $B_\beta(q)$ is infinite, a contradiction. For each $p\in X\setminus S$, use Lemma \ref{periodic or not} to choose $\gamma_p>0$ such that either (1) $|\{i\in\omega:p=x_i\}|\leq 1$ or (2) $p$ is periodic and there exists $M\in\omega$ with $x_M$ in the orbit of $p$ and for any such $M$, $\sup \{d({x_{M+i}},{f^i(x_M)})\}<\gamma/3$.

	Now, fix $\delta_0$ to be the lesser of $\beta$ and $\min\{\gamma_p:p\in X\setminus S\}$. We claim that every $\delta_0$-pseudo-orbit is $\epsilon$-shadowed by an orbit for $f$.
	
	In order to verify this, let $\seq{x_i}$ be a $\delta_0$-pseudo-orbit for $f$. We will now define a nearby $\gamma$-pseudo-orbit $\seq{y_i}$ for which $y_i=y_j$ implies $y_{i+1}=y_{j+1}$ to which we will apply Lemma \ref{functionally generated} to achieve our claimed results.
	
	There are two cases to consider.
	
	\textbf{Case 1:} $\seq{x_i}$ has the property that if $i<j$ with $x_i= x_j$, then $x_i\in S$. In this case, we proceed inductively, defining $y_0=x_0$ and, once $y_i$ is defined for all $i<n$, we define $y_n=x_n$ if $x_n\notin S$ and choose $y_n\in B_\beta(x_n)\cap S\setminus\{y_i:i<n\}$ otherwise. Clearly, $\sup\{d({y_i},{x_i})\}\leq\beta$ and for $i\in\omega$,
	\begin{align*}
		\d(f(y_i),y_{i+1})&\leq d(f(y_i),f(x_i))+d(f(x_i),x_{i+1})+d(x_{i+1},y_{i+1})\\
		&<\gamma/3+\delta_0+\beta<\gamma.
	\end{align*}
	i.e., $\seq{y_i}$ is a $\gamma$-pseudo-orbit. Furthermore, by construction if $i\neq j$, then $y_i\neq y_j$.

	\textbf{Case 2:} There exists $p\in X\setminus S$ and $m<n$ with $x_n=x_m=p$. Choose $M$ minimal so that $x_M$ is in the orbit of some point $p\notin S$ and there exists $M\leq m<n$ with $x_{m}=x_n=p$. By choice of $\delta_0$ and the existence of $m,n$ we are in case (2) of Lemma \ref{periodic or not}, and so $x_M$ is periodic and $\sup\{d({x_{M+i}},f^i(x_M))\}<\gamma/3$.

	For $i<M$, we choose $y_i$ as in Case 1, with the additional condition that $y_i$ not be in the orbit of $x_M$.

	For $i\geq M$, define $y_i=f^{i-M}(x_M)$. It is then clear that $\seq{y_i}$ is a $\gamma$-pseudo-orbit (by the same argument above) and that $\sup\{d({x_i},{y_i})\}\leq\max\{\beta,\gamma/3\}=\gamma/3$. By construction we see that the hypotheses of Lemma \ref{functionally generated} are met.
	
	In either case, we have a $\gamma$-pseudo-orbit $\seq{y_i}$ satisfying the hypotheses of Lemma \ref{functionally generated} such that $\sup\{d({x_i},{y_i})\}\leq\gamma/3$.

	By applying Lemma \ref{functionally generated}, we can find a function $g:X\to X$ with $\rho(g_N,f)\leq\gamma<\delta$ and $z\in X$ such that for all $i\in\omega$, $g^i(z)=y_i$.
	
	By our choice of $\delta$, there exists $q\in X$ such that the $f$-orbit of $q$ $\epsilon/2$-shadows the $g$-orbit of $z$.  Thus, for $i\in\omega$, we have
	\[d(f^i(q),x_i)\leq d(f^i(q),g^i(z))+d(g^i(z),y_i)+d(y_i,x_i)<\epsilon/2+0+\gamma/3<\epsilon.\]
	
	In other words, the $\delta_0$-pseudo-orbit $\seq{x_i}$ is $\epsilon$-shadowed by the $f$-orbit of $q$. Thus $f$ has shadowing as claimed.
\end{proof}

\section{Shadowing as a structural property in  $\mathcal C(X)$} \label{C(X)}

In this section, we develop results analogous to those of Section \ref{X^X} in the space of \emph{continuous} dynamical systems. It is worth highlighting that the results of that section depended heavily on freedom with which functions can be modified if continuity is not under consideration. However, modifying continuous functions appropriately is far more delicate. Broadly, if the domain under consideration allows for continuous functions to be perturbed with enough freedom, then we can demonstrate that shadowing is a structural property of the space of continuous dynamical systems.

We begin by defining the notion of perturbability.

\begin{definition}
	A metric space $X$ is \emph{perturbable} provided that for all $\epsilon>0$, there exists $\delta>0$ such that for every continuous function $f:X\to X$, every finite subset $F=\{x_i: i\leq N\}\subseteq X$, and every function $g_0:F\to X$ with $\rho(g_0,f|_{F})<\delta$, there exists a continuous function $g:X\to X$ with $\rho(f,g)<\epsilon$ and $g|_F=g_0$.
\end{definition}

It is worth noting that the collection of perturbable spaces is quite large---totally disconnected spaces are perturbable, for example. Spaces with a degree of generalized local convexity are also perturbable. 

\begin{definition}[\cite{Horvath}]
	A \emph{c-structure} on a topological space $X$ is a function $F$ from the set of nonempty finite subsets of $X$ into the collection of nonempty contractible subspaces of $X$ such that if $A\subseteq B$, then $F(A)\subseteq F(B)$.
	
	A nonempty subset $E$ of $X$ is an \emph{$F$-set} provided that $F(A)\subseteq E$ for all finite nonempty $A\subseteq E$.
	
	A space $X$ with c-structure $F$ is an \emph{m.c.-space} provided that each point has a neighborhood base consisting of $F$-sets.
\end{definition}

In the above definition, the function $F$ essentially plays the role in an m.c.-space that the convex hull plays in a locally convex space. Spaces which are m.c.-spaces include manifolds, topological graphs, and many other classes of continua.

The following extension theorem for m.c.-spaces is a generalization of Dugundji's extension theorem \cite{Dugundji} and  will be useful going forward.

\begin{theorem}[\cite{Horvath}]
	Let $X$ be a metric space, $A\subseteq X$ closed, $Y$ a m.c.-space with c-structure $F$ and $f:A\to Y$ continuous. Then there exists a continuous function $g:X\to Y$ which extends $f$ such that $g(X)\subseteq E$ for any $F$-set $E$ in $Y$ which contains $f(A)$.
\end{theorem}

In particular, we apply this theorem to establish the following result.

\begin{proposition} \label{lc is lp}
	If a compact metric space $X$ is totally disconnected or has a c-structure which makes it an m.c.-space, then $X$ is perturbable.
\end{proposition}

\begin{proof}
	Fix $X$, $\epsilon>0$ and $F\subseteq X$ finite. 
	
	First, suppose that $X$ has c-structure $F$ which makes it an m.c.-space.
	Since $X$ is compact and metric, we can choose $\delta<\epsilon/2$ such that for every point $x\in X$, the set $B_\delta(x)$ is an $F$-set
	
	Now, fix $f:X\to X$ continuous and choose $g_0:F\to X$ a function with $\rho(g_0,f|F)<\delta$. 	Since $f$ is continuous on the compact metric space $X$, choose $\eta>0$ such that if $d(a,b)<\eta$, then $d(f(a),f(b))<\delta$ and such that if $p,p'\in F$ with $p\neq p'$, $d(p,p')>\eta$. 
	
	Now, for each $p\in F$, let $X_p=\{x\in X:d(x,p)\leq\eta/2\}$ and let $C_p=\partial X_p\cup\{p\}$. Define $g_p:C_p\to X$ by
	\[h_p(x)=\begin{cases} 
	f(x) & x\in\partial X_p \\
	g_0(p) & x=p
	\end{cases}.\]
	
	By construction, for each $x\in C_p$, $d(x,p)<\eta$ and hence $d(f(x),f(p))<\delta$. Since $g_0(p)$ is also within $\delta$ of $f(p)$, we see that $h_p(C_p)\subseteq B_\delta(f(p))$. Since the latter is convex, and $C_p$ is a closed subset of $X_p$, by Theorem 2 of \cite{Horvath}, since $B_\delta(f(p))$ is an $F$-set containing $g_p(C_p)$, there is a continuous extension $g_p:X_p\to X$ of $h_p$ with $g_p(X_p)$ a subset $B_\delta(f(p))$.
	
	By choice of $\eta$, $\{X_p:p\in F\}$ is a collection of pairwise disjoint closed sets, and thus we can define the desired function $g:X\to X$ as follows
	\[g(x)=\begin{cases}
	g_p(x): x\in X_p\\
	f(x): x\notin \bigcup X_p.
	\end{cases}\]
	
	Clearly, for $p\in F$, $g(p)=g_p(p)=h_p(p)=g_0(p)$. To see that $\rho(g,f)<\epsilon$, observe that if $x\notin\bigcup X_p$, then $f(x)=g(x)$, and if $x\in X_p$, then $d(x,p)\leq\eta/2$ and hence $d(f(x), f(p))<\delta$. But $g(x)=g_p(x)$ is contained in $B_\delta(f(p))$ by construction, and hence $d(f(x),g(x))<2\delta<\epsilon$.
	
	If $X$ is instead totally disconnected, we proceed as follows. First, choose $\delta<\epsilon$.  
	
	Now, fix $f:X\to X$ and $g_0:F\to X$ with $\rho(g_0,f|F)<\delta$. By uniform continuity, choose $\eta>0$ such that if $d(a,b)<\eta$, then $d(f(a),f(b))<\epsilon-\delta$. Since $X$ is compact and totally disconnected, for each $p\in F$, choose a neighborhood $X_p$ of $p$ with diameter less than $\eta$ and with empty boundary and such that $\{X_p:p\in F\}$ is a collection of pairwise disjoint sets. Now, we define the desired function $g:X\to X$ by:
	
	\[g(x)=\begin{cases}
	g_0(x): x\in X_p\\
	f(x): x\notin \bigcup X_p.
	\end{cases}\]
	
	By construction, $g$ is continuous, maps $p$ to $g_0(p)$ for all $p\in F$ and for $x\notin\bigcup X_p$, $d(f(x),g(x))=0<\epsilon$. For $x\in X_p$, we see that $d(x,p)<\eta$ and hence $d((f(x),f(p))<\epsilon-\delta$. Since $d(f(p),g(p))=d(f(p),g_0(p))<\delta$, it follows that $d(f(x),g(x))<\epsilon$, and thus we see $\rho(f,g)<\epsilon$.
\end{proof}

For the purposes of establishing analogs to Theorem \ref{shadowing in sequences} and Corollary \ref{shadowing in sequences corollary}, a weaker form of perturbability is sufficient.

\begin{definition}
	A metric space $X$ is \emph{weakly perturbable} provided that for all $\epsilon>0$, there exists $\delta>0$ such that for every continuous function $f:X\to X$ and every pair $(x,y)\in X\times X$ with $d(y,f(x))<\delta$, there exists a continuous function $g:X\to X$ with $\rho(f,g)<\epsilon$ and $g(x)=y$.
\end{definition}

It is immediately clear that every perturbable space is weakly perturbable. It is not difficult to see that the converse statement is false---the product of the interval and the Cantor set is weakly perturbable but not perturbable.

\begin{lemma} \label{continuous}
	Suppose $X$ is weakly perturbable. Let $\seq{f_i}\in \mathcal C(X)^\omega$ and $\epsilon>0$. Then there exists $\delta>0$ such that if $\seq{x_i}$ is a $\delta$-pseudo-orbit for $\seq{f_i}$, then there exists $\seq{g_i}\in \mathcal C(X)^\omega$  such that $\rho(\seq{f_i},\seq{g_i})\leq\epsilon$ and $\seq{x_i}$ is an orbit for $\seq{g_i}$ .
\end{lemma}

\begin{proof}
Fix $\epsilon>0$ and choose $\delta>0$ to witness the weak perturbability of $X$. Let $\seq{x_i}$ be a $\delta$-pseudo-orbit for $\seq{f_i}$.

For each $i\in\omega$, notice that $d(f_i(x_i),x_{i+1})<\delta$, and thus we can apply weak perturbability to choose a function $g_i\in\mathcal C(X)$ with $g_i(x_i)=x_{i+1}$ and $\rho(g_i,f_i)<\epsilon$. It is then immediate that $\seq{x_i}$ is an orbit for $\seq{g_i}$ and $\rho(\seq{f_i},\seq{g_i})\leq\epsilon$.
\end{proof}

With this result, Theorem \ref{sequence continuous} and Corollary \ref{single sequence continuous} follow by the same argumentation as Theorem \ref{shadowing in sequences} and Corollary \ref{shadowing in sequences corollary}, respectively.

\begin{theorem} \label{sequence continuous}
	Let $X$ be a space which is weakly perturbable and let $\seq{f_i}\in\mathcal C(X)^\omega$. The dynamical system $(\seq{f_i},X)$ has shadowing if and only if for all $\epsilon>0$, there exists $\delta>0$ such that if $\seq{g_i}\in\mathcal C(X)^\omega$ with $\rho(\seq{g_i},\seq{f_i})<\delta$, then each $\seq{g_i}$-orbit  is $\epsilon$-shadowed by an $\seq{f_i}$-orbit.
\end{theorem}

\begin{corollary} \label{single sequence continuous}
	Let $X$ be a space which is weakly perturbable and let $f\in\mathcal C(X)$. The dynamical system $(X,{f})$ has shadowing if and only if for all $\epsilon>0$, there exists $\delta>0$ such that if $\seq{g_i}\in\mathcal C(X)^\omega$ with $\rho(\seq{g_i},\seq{f})<\delta$, then each $\seq{g_i}$-orbit  is $\epsilon$-shadowed by an ${f}$-orbit.
\end{corollary}

In order to prove an analog for Theorem \ref{shadowing in X^X}, it would be desirable to have an analog for Lemma \ref{functionally generated}, i.e. the ability to generate a pseudo-orbit as an orbit of a nearby function. In the previous section, this was easily achievable, as continuity was not a concern. However, for this result, we would want the nearby function to be continuous. In general, this would likely require being able to perturb a function on an \emph{infinite} set, and (since we are requiring continuity) hence also on any limit points. However, in the compact setting, we will see that being able to perturb a function on finitely many points is sufficient.

\begin{lemma} \label{finite approx}
	Suppose $X$ is perturbable . Let $f\in\mathcal C(X)$ and $\delta>0$. Then there exists $\gamma>0$ such that if $\seq{x_i}$ a $\gamma$-pseudo-orbit for $f$ for which $x_i=x_j$ implies that $x_{i+1}=x_{j+1}$, then for every natural number $N$, there exists $g\in\mathcal C(X)$ with $\rho(g,f)<\delta$ and $z\in X$  such that for all $i\leq N$, $g^i(z)=x_i$.
\end{lemma}

\begin{proof}
	Fix $\delta>0$ and let $\gamma>0$ witness the perturbability of $X$ and let $\seq{x_i}$ be a $\gamma$-pseudo-orbit for $f$ satisfying the hypotheses of the lemma. 
	
	Now, fix $N$ and define $F=\{x_i:i< N\}$ and $g_0:F\to X$ by $g_0(x_i)=x_{i+1}$. Note that $g_0$ is well-defined since if $x_i=x_j$, then $g_0(x_i)=x_{i+1}=x_{j+1}=g_0(y_i)$. By perturbability of $X$ and choice of $\gamma$, there exists a continuous $g_N:X\to X$ such that $\rho(g_N,f)<\delta$ and for $z=x_0$ we have $g^i(z)=x_i$ for $i\leq N$, as claimed.
\end{proof}

The following lemma is very much a natural analog to Lemma \ref{periodic or not} and the proof differs only in that it appeals to Lemma \ref{periodic or nonrepeating} rather than Lemma \ref{periodic or not}. We include a full proof for completeness.

\begin{lemma} \label{periodic or nonrepeating}
	Let $X$ be a perturbable compact metric space and let $f\in C(X)$ such that for every $\epsilon>0$ there exists $\delta>0$ such that if $g\in \mathcal C(X)$ with $\rho(g,f)<\delta$, then each $g$-orbit  is $\epsilon$-shadowed by an $f$-orbit. If $p\in X$ is isolated, then for every $\delta'$ there exists $\gamma_p>0$ such that if $\seq{x_i}$ is a $\gamma_p$-pseudo-orbit for $f$, then either
	\begin{enumerate}
		\item $|\{i\in\omega:p=x_i\}|\leq 1$ or
		\item $p$ is periodic and there exists $M\in\omega$ with $x_M$ in the orbit of $p$ and for any such $M$, $\sup\{d({x_{M+i}},{f^i(x_M)})\}<\delta'$.
	\end{enumerate}

\end{lemma}

\begin{proof}
	Fix $\delta'>0$. Let $p\in X$ be isolated and suppose that for every $\gamma>0$ there exists a $\gamma$-pseudo-orbit for $f$ in which $p$ appears at least twice (else we can choose $\gamma_p>0$ such that condition (1) holds for every $\gamma_p$-pseudo-orbit).
	Since $p$ is isolated, we can choose $\epsilon>0$ with $B_\epsilon(p)=\{p\}$, and by supposition, there is $\delta>0$ such that if $\rho(g,f)<\delta$, then every $g$-orbit is $\epsilon$-shadowed by an $f$-orbit. Finally, choose $\delta>\gamma>0$ satisfying the conclusion of Lemma \ref{finite approx}.
	
	By supposition, we can find a $\gamma$-pseudo-orbit $\seq{x_i}$ for $f$ in which $p$ appears at least twice---let $M\in\omega$ and $N>0$ be minimally chosen so that $p=x_M=x_{N+M}$.
	
	Define, for $i<N$, $y_i=x_{M+i}$ and consider the finite sequence $y_0,\ldots, y_{N-1}$. Notice that $y_0=p$, $d(f(y_{N-1}),y_0)<\gamma$, and for all $i<N-1$, we have $d(f(y_i),y_{i+1})<\gamma$. Suppose that there is an element of $X$ that occurs more than once in this list. Choose $n<N-1$ and $t>0$ such that $y_n=y_{n+t}$ and define the finite sequence $y'_0,\ldots y'_{N'-1}$ as follows. Define $N'=N-t$ and 
	\[ y'_i=\begin{cases} 
		y_i & i\leq n \\
		y_{i+t} & i>n 
	\end{cases}
	\]
	and observe that $y'_0=p$, $d(f(y'_{N-1}),z_0)<\gamma$, and  for all $i<N'-1$, we have $d(f(y'_i),y'_{i+1})<\gamma$. By iterating this process as necessary, we arrive at a finite sequence $z_0,\ldots z_{P-1}$ with no repeated elements, such that $z_0=p$, $d(f(z_{P-1}), z_0)<\gamma$, and for $i<P-1$, $d(f(z_{i}), z_{i+1})<\gamma$. It is easy to see then that if we define $z_i=z_{i\mod P}$ for $i\geq P$,  $\seq{z_i}$ is a $\gamma$-pseudo-orbit with $z_{kP}=p$ for all $k\in\omega$ which satisfies the hypotheses of Lemma \ref{finite approx}.
	
	As such, we can find a function $g\in\mathcal C(X)$ with $\rho(g,f)<\delta$ such that $\seq{z_i}$ is a $g$-orbit, and therefore is $\epsilon$-shadowed by the $f$-orbit of some point. But that point and its image under $f^P$ must lie in $B_\epsilon(p)=\{p\}$, and therefore the shadowing point is $p$ itself and $p$ has period $P$.
	
	Now, by uniform continuity of $f, f^2, \ldots f^{P-1}$, choose $\gamma_p>0$ such that $\gamma_p<\min\{\epsilon,\delta'/2\}/P$ and if $d(a,b)<\gamma_p$, then $d(f^i(a),f^i(b))<\min\{\epsilon,\delta'/2\}/P$ for all $i<P$. Now, let $\seq{x_i}$ be a $\gamma_p$-pseudo-orbit for $f$ with $x_M$ in the orbit of $p$ with $k<P$ chosen such that $f^k(x_M)=p$. Then, for $i\leq P$, we have 
	\begin{align*}
		d(x_{M+i},f^i(x_M))&\leq \sum_{j=0}^{i-1} d\left(f^{(i-j)}(x_{M+j}),f^{(i-j-1)}(x_{M+j+1})\right)\\
		&=d(f(x_{M+i-1}),x_{M+i})+\sum_{j=0}^{i-2} d\left(f^{(i-j)}(f(x_{M+j})),f^{(i-j)}(x_{M+j+1})\right)\\
		&<\gamma_p+\sum_{j=0}^{i-2} \min\{\epsilon,\delta'/2\}/P<\min\{\epsilon,\delta'/2\}
	\end{align*}

	In particular, $x_{M+k}\in B_\epsilon(f^k(x_M))=B_\epsilon(p)=\{p\}$ and by induction (applying the above argument to $x_{M+k}=p$), we can establish that $d(x_{M+k+i},f^i(p))<\delta'/2$ for all $i\in\omega$, and thus $\sup\{d({x_{M+i}},f^i(p))\}<\delta'$ as desired.

\end{proof}

With these tools in hand, we can now prove that, in perturbable spaces, shadowing is a structural property of $\mathcal C(X)$. As with the previous lemma, this theorem and its proof are very analogous to Theorem \ref{shadowing in X^X}---again, we include the full proof for completeness.

\begin{theorem} \label{perturbable}
	Let $X$ be a perturbable compact metric space and let $f\in C(X)$. The dynamical system $(X,f)$ has shadowing if and only if for every $\epsilon>0$ there exists $\delta>0$ such that if $g\in \mathcal C(X)$ with $\rho(g,f)<\delta$, then each $g$-orbit  is $\epsilon$-shadowed by an $f$-orbit.
\end{theorem}

\begin{proof}
	As in the proof of Theorem \ref{shadowing in sequences}, if $(X,f)$ has shadowing, then the conclusion follows immediately.
	
	So, let us now assume that $(X,f)$ has the indicated property and fix $\epsilon>0$. Choose $\epsilon/4>\delta>0$ to witness that if $g
	\in\mathcal C(X)$ with $\rho(f,g)<\delta$, then each $g$-orbit is $\epsilon/4$ shadowed by an $f$-orbit. Now, by Lemma \ref{finite approx}, choose $\delta>\gamma>0$ such that if $\seq{x_i}$ is a $\gamma$-pseudo-orbit for $f$ for which $x_i=x_j$ implies $x_{i+1}=x_{j+1}$, then the function $g$ in the conclusion of that lemma satisfies $\rho(g,f)<\delta$.

	Now, by uniform continuity of $f$, fix $\gamma/3>\beta$ such that if $d(a,b)<\beta$, then $d(f(a),f(b))<\gamma/3$, and define
	\[S=\{x\in X:B_\beta(x) \textrm{ is infinite}\}.\]
	
	Observe that $X\setminus S$ is finite---otherwise, it would have a limit point $p$ and a point $q\in X\setminus S$ with $p\in B_\beta(q)$. But then $B_\beta(q)$ is infinite, a contradiction. For each $p\in X\setminus S$, use Lemma \ref{periodic or nonrepeating} to choose $\gamma_p>0$ such that either (1) $|\{i\in\omega:p=x_i\}|\leq 1$ or (2) $p$ is periodic and there exists $M\in\omega$ with $x_M$ in the orbit of $p$ and for any such $M$, $\sup\{d({x_{M+i}},{f^i(x_M)})\}<\gamma/3$.

Now, fix $\delta_0$ to be the lesser of $\beta$ and $\min\{\gamma_p:p\in X\setminus S\}$ and let $\seq{x_i}$ be a $\delta_0$-pseudo-orbit for $f$. We will now define a nearby $\gamma$-pseudo-orbit $\seq{y_i}$ for which $y_i=y_j$ implies $y_{i+1}=y_{j+1}$ to which we will apply Lemma \ref{finite approx} to achieve our claimed results.

There are two cases to consider.

\textbf{Case 1:} $\seq{x_i}$ has the property that if $i<j$ with $x_i= x_j$, then $x_i\in S$. In this case, we proceed inductively, defining $y_0=x_0$ and, once $y_i$ is defined for all $i<n$, we define $y_n=x_n$ if $x_n\notin S$ and choose $y_n\in B_\beta(x_n)\cap S\setminus\{y_i:i<n\}$ otherwise. Clearly, $\sup\{d({y_i},{x_i})\}\leq\beta$ and for $i\in\omega$,
\begin{align*}
	\d(f(y_i),y_{i+1})&\leq d(f(y_i),f(x_i))+d(f(x_i),x_{i+1})+d(x_{i+1},y_{i+1})\\
	&<\gamma/3+\delta_0+\beta<\gamma.
\end{align*}
i.e., $\seq{y_i}$ is a $\gamma$-pseudo-orbit. Furthermore, by construction if $i\neq j$, then $y_i\neq y_j$.

\textbf{Case 2:} There exists $p\in X\setminus S$ and $m<n$ with $x_n=x_m=p$. Choose $M$ minimal so that $x_M$ is in the orbit of some point $p\notin S$ and there exists $M\leq m<n$ with $x_{m}=x_n=p$. By choice of $\delta_0$ and the existence of $m,n$ we are in case (2) of Lemma \ref{periodic or nonrepeating}, and so $x_M$ is periodic and $\sup\{d({x_{M+i}},f^i(x_M))\}<\gamma/3$.

For $i<M$, we choose $y_i$ as in Case 1, with the additional condition that $y_i$ not be in the orbit of $x_M$.

For $i\geq M$, define $y_i=f^{i-M}(x_M)$. It is then clear that $\seq{y_i}$ is a $\gamma$-pseudo-orbit (by the same argument above) and that $\sup\{d({x_i},{y_i})\}\leq\max\{\beta,\gamma/3\}=\gamma/3$. By construction we see that the hypotheses of Lemma \ref{finite approx} are met.

In either case, we have a $\gamma$-pseudo-orbit $\seq{y_i}$ satisfying the hypotheses of Lemma \ref{finite approx} such that $\sup\{d({x_i},{y_i})\}\leq\gamma/3$.

	By applying Lemma \ref{finite approx}, for each $N$, we can fix a continuous $g_N:X\to X$ with $\rho(g_N,f)<\delta$ and $z_N\in X$ such that for all $i\leq N$, $g^i(z_N)=y_i$.
	
	By our choice of $\delta$, for each $N$, there exists $q_N\in X$ such that the $f$-orbit of $q_N$ $\epsilon/4$-shadows the $g$-orbit of $z_N$.  Thus, for $i\leq N$, we have
	\[d(f^i(q_N),x_i)\leq d(f^i(q_N),g_N^i(z_N))+d(g_N^i(z_N),y_i)+d(y_i,x_i)<\epsilon/4+0+\gamma/3<\epsilon/2.\]
	
	By passing to a subsequence if necessary, let $q=\lim_{N\to\infty} q_N$ and observe that for all $i\in\omega$, we have $d(f^i(q),x_i)=\lim_{N\to\infty}d(f^i(q_N),x_i)\leq\epsilon/2<\epsilon$, i.e. the $\delta_0$-pseudo-orbit $\seq{x_i}$ is $\epsilon$-shadowed by the $f$-orbit of $q$. Thus $f$ has shadowing as claimed.
\end{proof}

Again, by Proposition \ref{lc is lp}, we have the following corollary.

\begin{corollary}
	Let $X$ be a compact metric space which is either totally disconnected or  has a c-structure which makes it an m.c.-space and let $f\in C(X)$. The dynamical system $(X,f)$ has shadowing if and only if for every $\epsilon>0$ there exists $\delta>0$ such that if $g\in \mathcal C(X)$ with $\rho(g,f)<\delta$, then each $g$-orbit  is $\epsilon$-shadowed by an $f$-orbit.
\end{corollary}

The notions of perturbability above are properties of the space under consideration, however they can easily be rephrased as properties of individual functions.

\begin{definition}
	A function $f:X\to X$ on a metric space $X$ is \emph{perturbable} provided that for all $\epsilon>0$, there exists $\delta>0$ such that for every finite subset $F=\{x_i: i\leq N\}\subseteq X$ and every function $g_0:F\to X$ with $\rho(g_0,f|_{F})<\delta$, there exists a continuous function $g:X\to X$ with $\rho(f,g)<\epsilon$ and $g|_F=g_0$.
\end{definition}

Replacing the perturbability of $X$ with perturbability of $f$ in Lemma \ref{finite approx} and Theorem \ref{perturbable} yields the following.

\begin{theorem}
	Let $X$ be a compact metric space and $f\in C(X)$ be perturbable. The dynamical system $(X,f)$ has shadowing if and only if for every $\epsilon>0$ there exists $\delta>0$ such that if $g\in \mathcal C(X)$ with $\rho(g,f)<\delta$, then each $g$-orbit  is $\epsilon$-shadowed by an $f$-orbit.
\end{theorem}

\section{Other Results} \label{other}

There has been significant recent interest in the study of variations of the shadowing property \cite{BMR-Variations, sakai, Oprocha-Thick, OprochaDastjerdiHosseini}. Broadly, many of these variations arise by expanding or restricting the class of `pseudo-orbits' that need to be shadowed (or the degree to which they need to be shadowed). By Remark \ref{nearby}, if $(X,f)$ is a dynamical system and $g$ is a nearby function, then the orbits of $g$ are pseudo-orbits for $f$, and these pseudo-orbits are a proper subset of the full collection of pseudo-orbits. This observation motivates the following variations of the shadowing property.

\begin{definition}
	For $f\in X^X$ and $\delta>0$, a \emph{functionally generated $\delta$-pseudo-orbit} is a sequence $\seq{x_i}$ such that there exists $g\in X^X$ with $\rho(f,g)<\delta$ and $z\in X$ such that $x_i=g^i(z)$ for all $i\in \omega$.
\end{definition}

Functionally generated pseudo-orbits are a very natural type of pseudo-orbit to consider. They arise naturally in finite precision computations of orbits, for example.
Notice that the functionally generated pseudo-orbits are a proper subset of the collection of all pseudo-orbits, and thus the following notion of shadowing is, \emph{prima facie}, weaker than the standard shadowing property.

\begin{definition}
	A dynamical system $(X,f)$ has the \emph{finitely generated pseudo-orbit tracing property (FGPOTP)} provided that for all $\epsilon>0$ there exists $\delta>0$ such that every functionally generated $\delta$-pseudo-orbit $\seq{x_i}$, there exists $z\in X$ such that the orbit of $z$ under $f$ $\epsilon$-shadows $\seq{x_i}$.
\end{definition}

Of course, we can restrict the class of pseudo-orbits under consideration further by requiring them to be generated by nearby continuous maps.

\begin{definition}
	For $f\in X^X$ and $\delta>0$, a \emph{continuously generated $\delta$-pseudo-orbit} is a sequence $\seq{x_i}$ such that there exists $g\in \mathcal C(X)$ with $\rho(f,g)<\delta$ and $z\in X$ such that $x_i=g^i(z)$ for all $i\in \omega$.
\end{definition}

Restricting our class of pseudo-orbits to the continuously generated ones gives rise to the following even weaker notion of shadowing.

\begin{definition}
	A dynamical system $(X,f)$ has the \emph{continuously generated pseudo-orbit tracing property (CGPOTP)} provided that for all $\epsilon>0$ there exists $\delta>0$ such that every functionally generated $\delta$-pseudo-orbit $\seq{x_i}$, there exists $z\in X$ such that the orbit of $z$ under $f$ $\epsilon$-shadows $\seq{x_i}$.
\end{definition}

With these notions in mind, Theorems \ref{shadowing in X^X} and Theorem \ref{perturbable} immediately give us the following result.

\begin{corollary}
	Let $f\in \mathcal C(X)$. Then
	\begin{enumerate}
		\item $f$ has shadowing if and only if it has FGPOTP.
		\item If $f$ has FGPOTP, then it has CGPOTP.
		\item If $X$ is perturbable and $f$ has CGPTOP, then it has shadowing and hence FGPOTP.
	\end{enumerate}
\end{corollary}

Of course, it would be interesting to know how `sharp' this result is.

\begin{question} \label{First Q}
	Are there systems which have CGPOTP but not FGPOTP?
\end{question}

Another interesting viewpoint of the results from Section \ref{C(X)} comes about by considering the \emph{set} of orbits generated by a dynamical system. Orbits for $(X,f)$ belong to the set $X^\omega$, and if we wish to consider shadowing-type properties, the natural metric on $X^\omega$ to consider is the \emph{uniform metric}, i.e. \[\rho(\seq{x_i},\seq{y_i})=\sup_{i\in\omega}\{d(x_i,y_i)\}.\]
We can then define the function $\mathcal O:\mathcal C(X)\to X^\omega$ which maps a function to its set of orbits, i.e. $\mathcal O(f)=\{\seq{f^i(x)}:x\in X\}.$

When considering set-valued functions between topological spaces, there are several distinct generalizations of continuity. For the purposes of this paper, we consider the notion of \emph{upper Hausdorff semicontinuity}. It is worth mentioning that for non-compact codomains, upper Hausdorff semicontinuity is generally weaker than upper semicontinuity \cite{DoleckiRolewiczUSC}.

\begin{definition}
	A set-valued function $g:A\to B$ between metric spaces $(A,d_A)$ and $(B,d_B)$ is \emph{upper Hausdorff semicontinuous} at $a\in A$ provided that for all $\epsilon>0$ there exists $\delta>0$ such that if $d_A(a,x)<\delta$, then $g(x)\subseteq B_\epsilon(g(a))$.
\end{definition}

Using Theorem \ref{perturbable}, we can prove the following result.

\begin{theorem} \label{usc}
	Let $X$ be a compact metric space  which is perturbable and let $f\in C(X)$. Then $f$ has shadowing if and only if $\mathcal O:\mathcal C(X)\to X^\omega$ is upper Hausdorff semicontinuous at $f$.
\end{theorem}

\begin{proof}
First, assume that $f$ has shadowing. Fix $\epsilon>0$ and let $\delta>0$ witness that every $\delta$-pseudo-orbit is $\epsilon/2$-shadowed. Observe that if $\rho(f,g)<\delta$,  and $\seq{x_i}\in\mathcal O(g)$, then by Remark \ref{nearby}, $\seq{x_i}$ is a $\delta$-pseudo-orbit for $f$, and as such, there exists $z\in X$ such that the $f$-orbit of $z$ $\epsilon/2$-shadows $\seq{x_i}$. But then $\rho(\seq{x_i},\seq{f^i(z)})\leq\epsilon/2<\epsilon$, and since $\seq{f^i(z)}\in\mathcal O(f)$, we have that $\seq{x_i}\in B_\epsilon(\mathcal O(f))$. Thus $\mathcal O$ is upper Hausdorff semicontinuous at $f$.

Now, suppose that $\mathcal O$ is upper Hausdorff semicontinuous at $f$. Fix $\epsilon>0$ and choose $\delta>0$ so that if $\rho(f,g)<\delta$, then $\mathcal O(g)\subseteq B_{\epsilon}(\mathcal O(f))$ . Let $g\in\mathcal C(X)$ with $\rho(f,g)<\delta$ and let $x\in X$. Then $\seq{g^i(x)}\in\mathcal O(g)\in B_{\epsilon}(\mathcal O(f))$. In particular, there exists $z\in X$ such that $\rho(\seq{g^i(x)},\seq{f^i(z)})<\epsilon$. In particular, for each $i\in\omega$, $d(g^i(x),f^i(z))<\epsilon$, i.e. the $g$-orbit of $x$ is $\epsilon$-shadowed by the $f$-orbit of $z$. Thus, by Theorem \ref{perturbable}, $f$ has shadowing.
\end{proof}

\subsection*{Acknowledgements} The author would like to thank Professor Jim Wiseman for helpful feedback and discussion in the revision of this paper.

\bibliographystyle{plain}
\bibliography{../../ComprehensiveBib}

\end{document}